\documentclass[11pt]{article}

\usepackage[margin=1.08in]{geometry}
\usepackage{amsmath,amssymb,amsthm}
\usepackage{microtype}
\usepackage{needspace}
\usepackage{xcolor}
\usepackage[colorlinks=true,linkcolor=blue!55!black,citecolor=blue!55!black,
            urlcolor=blue!60!black]{hyperref}

\newtheorem{theorem}{Theorem}[section]
\newtheorem{corollary}[theorem]{Corollary}
\newtheorem{lemma}[theorem]{Lemma}
\theoremstyle{definition}

\theoremstyle{remark}
\newtheorem{remark}[theorem]{Remark}

\newcommand{\N}{\mathbb{Z}_{>0}}
\newcommand{\Pcal}{\mathcal{P}}
\newcommand{\Bcal}{\mathcal{B}}
\newcommand{\vth}{\vartheta}

\title{Refuting a Conjecture of Umans and Wang on Arithmetic-Progression Divisor Covers}
\author{
Xinjie He\thanks{University of California, Los Angeles.
  Email: \href{mailto:xinjieh@math.ucla.edu}{\texttt{xinjieh@math.ucla.edu}}.}
  \qquad
  Amit Sahai\thanks{University of California, Los Angeles.
  Email: \href{mailto:sahai@cs.ucla.edu}{\texttt{sahai@cs.ucla.edu}}.}
  }
\date{\today}

\begin{document}
\maketitle

\begin{abstract}
An \emph{$n$-divisor set} is a finite set of positive integers containing
a multiple of every integer from $1$ through $n$.  Umans and Wang proposed,
as the arithmetic-progression version of their Strong $(\alpha,\beta)$-Divisor
Conjecture, an $n$-divisor arithmetic progression having at most
$n^{2\beta}$ terms, each of magnitude at most $\exp(n^\alpha)$.  We prove
unconditionally that an $n$-divisor arithmetic progression of height $H$
with $\log H=o(\sqrt n)$ must have length
\[
 L\ge
 \left(\sqrt{\frac{8}{27}}-o(1)\right)
 \frac{n^{3/4}}{\sqrt{\log n}}.
\]
Consequently, the arithmetic-progression version is false whenever
$\alpha<1/2$ and $\beta<3/8$.  In particular, it is false at the proposed
point $(\alpha,\beta)=(1/3,1/3)$, even if both bounds are relaxed by
$n^{o(1)}$ at the exponent level.  The proof uses primes in a fixed band below
$\sqrt n$ to turn semiprime divisibility into a finite incidence structure.
An elementary bounded-degree linear-space estimate then gives the result.
This theorem concerns the one-dimensional arithmetic-progression version
only; it does not disprove the higher-rank Strong Divisor Conjecture.
\end{abstract}

\section{Introduction and statement of results}

For $n\in\N$, a finite set $A\subset\N$ has the \emph{$n$-divisor
property} if every $d\in\{1,\ldots,n\}$ divides at least one member of $A$.
Umans and Wang introduced conjectures asking for $n$-divisor sets with
additional difference-set structure and small height
\cite{UmansWang2025}.  Proposition~3.4 and Section~6 of their paper single
out a formally stronger sufficient statement: for infinitely many $n$,
there should be a positive arithmetic progression
\[
                A=\{b+ic:0\le i<L\},
                \qquad L\le n^{2\beta},
\]
whose members have magnitude at most $\exp(n^\alpha)$ and which has the
$n$-divisor property.  They call this the \emph{Arithmetic Progression
Version} of the Strong $(\alpha,\beta)$-Divisor Conjecture.  It is an
informally named version rather than a separately numbered conjecture.

Our main theorem gives a quantitative obstruction to every such
progression of sub-square-root logarithmic height.

\newpage
\begin{theorem}[Quantitative AP lower bound]\label{thm:quantitative}
Let $n$ tend to infinity through any unbounded sequence, and for each such
$n$ let
\[
       A_n=\{b_n+i c_n:0\le i<L_n\}\subset\N
\]
be an arithmetic progression with $b_n,c_n\in\mathbb Z$ and with the
$n$-divisor property.  Put
$H_n=\max A_n$.  If
\[
                         \log H_n=o(\sqrt n),
\]
then
\begin{equation}\label{eq:main-lower-bound}
       L_n\ge
       \left(\sqrt{\frac{8}{27}}-o(1)\right)
       \frac{n^{3/4}}{\sqrt{\log n}}.
\end{equation}
Equivalently, for every $\varepsilon>0$, inequality
\eqref{eq:main-lower-bound} holds with
$\sqrt{8/27}-\varepsilon$ in place of
$\sqrt{8/27}-o(1)$ for all sufficiently large $n$ in the sequence.
\end{theorem}

\begin{corollary}[Exponent exclusion]\label{cor:exponent-exclusion}
Fix $\alpha,\beta\ge0$ with
\[
                  \alpha<\frac12,
          \qquad  \beta<\frac38.
\]
There is no unbounded sequence of integers $n$ admitting positive
$n$-divisor arithmetic progressions satisfying
\[
       L_n\le n^{2\beta+o(1)},
       \qquad
       H_n\le \exp\!\bigl(n^{\alpha+o(1)}\bigr).
\]
In particular, for every sufficiently large $n$ there is no such
progression under the stronger literal bounds
$L_n\le n^{2\beta}$ and $H_n\le\exp(n^\alpha)$.
\end{corollary}

Taking $\alpha=\beta=1/3$ disproves the Arithmetic Progression Version at
the one-third point, including its exponent-level relaxation.  The proof is
unconditional.  Its only analytic input is the prime number theorem, used
in the standard equivalent forms
\[
       \pi(x)\sim\frac{x}{\log x},
       \qquad
       \vth(x):=\sum_{p\le x}\log p\sim x;
\]
see, for example, \cite[Chapter 6]{MontgomeryVaughan2007}.  The
finite-incidence estimate is proved below.  It is related in spirit to
classical finite-linear-space arguments such as those of de Bruijn and
Erd\H{o}s \cite{deBruijnErdos1948}, but no external incidence theorem is
required.

\section*{Discovery using Codex}

The proof was discovered in an OpenAI Codex run using the
\texttt{gpt-5.6-sol} model with reasoning effort set to \texttt{ultra};
see \cite{OpenAIGPT56}.  Codex also produced the initial write-up.
Subsequent human review verified
the proof, reviewed the citations, and revised the exposition.  All named
authors approve the final manuscript and take responsibility for its claims
and any remaining errors.

\section*{Acknowledgments}

The prompting strategy for the Codex discovery run borrowed several design
ideas from the UCLA Moonshot Harness project \cite{ZhangEtAl2026Moonshot}.

\section*{Reproducibility}

The mathematical argument is self-contained and does not rely on
computer-assisted calculations or on access to the Codex transcript.  The
model identifier and reasoning setting of the discovery run are recorded in
the preceding disclosure.

\section{A bounded-degree linear-space lemma}

By a \emph{block} we mean an indexed proper subset $B_j\subsetneq V$ in an
indexed family $\Bcal=(B_j)_{j\in J}$.  Blocks with distinct indices may be
identical as subsets.

\Needspace{15\baselineskip}
\begin{lemma}[Bounded-degree linear cover]\label{lem:linear-cover}
Let $V$ be a finite set of $v\ge2$ points, and let
$\Bcal=(B_j)_{j\in J}$ be an indexed family of blocks.
Suppose that
\begin{enumerate}
  \item for every two distinct points $p,q\in V$, there exists a block
    $B_j$ containing both $p$ and $q$;
  \item blocks with distinct indices intersect in at most one point; and
  \item every point lies in at most $\Delta$ blocks.
\end{enumerate}
Then
\begin{equation}\label{eq:linear-cover-bound}
                         v\le \Delta(\Delta-1)+1.
\end{equation}
\end{lemma}

\begin{proof}
First fix a block $B_j$.  Since it is proper, choose
$u\in V\setminus B_j$.  For every $q\in B_j$, the pair-covering assumption
supplies a block containing $u$ and $q$.  The blocks obtained for distinct
$q$ have distinct indices: otherwise one of them would meet $B_j$ in at
least two points.  All these blocks contain $u$, so $|B_j|\le\Delta$.

Now fix $u\in V$.  The at most $\Delta$ blocks containing $u$ together
contain every other point of $V$.  Each such block has at most $\Delta-1$
points other than $u$.  Therefore
\[
                         v-1\le\Delta(\Delta-1),
\]
which is \eqref{eq:linear-cover-bound}.
\end{proof}

\section{Primes just below the square-root scale}

We first prove a lower bound associated with any fixed prime band.  The
constant in Theorem~\ref{thm:quantitative} will follow by optimizing the
band.

\Needspace{10\baselineskip}
\begin{lemma}[A fixed-band lower bound]\label{lem:fixed-band}
Under the hypotheses of Theorem~\ref{thm:quantitative}, fix real constants
$0<a<b\le1$.  Then
\begin{equation}\label{eq:fixed-band-lower}
       L_n\ge
       \left(a\sqrt{2(b-a)}-o(1)\right)
       \frac{n^{3/4}}{\sqrt{\log n}}.
\end{equation}
\end{lemma}

\begin{proof}
Write $x=\sqrt n$, and suppress the subscript $n$.  A decreasing
nonconstant progression can be reversed, so it may be written
\[
       A=\{u+i c:0\le i<L\},
       \qquad u,c\in\N,
       \qquad u+(L-1)c=H.
\]
If the progression has only one distinct term, then the $n$-divisor property
implies
\[
                 \operatorname{lcm}(1,\ldots,n)\mid H.
\]
The prime number theorem gives
\(
 \log\operatorname{lcm}(1,\ldots,n)=(1+o(1))n,
\)
contradicting $\log H=o(\sqrt n)$.  We may therefore assume $L\ge2$ and
$c\ge1$.  In particular, $c\le H$.

Let
\[
       \Pcal_0=\{p\text{ prime}:a x\le p\le b x\},
       \qquad
       \Pcal=\{p\in\Pcal_0:p\nmid c\}.
\]
By the prime number theorem,
\begin{align}
       |\Pcal_0|
          &=\left(b-a+o(1)\right)\frac{x}{\log x},
          \label{eq:band-count-before}\\
       \sum_{p\in\Pcal_0}\log p
          &=\left(b-a+o(1)\right)x.
          \label{eq:band-mass-before}
\end{align}
The number of primes in $\Pcal_0$ dividing $c$ is at most
\[
       \frac{\log c}{\log(a x)}
          \le \frac{\log H}{\log(a x)}
          =o\!\left(\frac{x}{\log x}\right),
\]
and
\[
  \sum_{\substack{p\in\Pcal_0\\ p\mid c}}\log p
  =\log\!\left(\prod_{\substack{p\in\Pcal_0\\ p\mid c}}p\right)
  \le \log c\le\log H=o(x).
\]
Consequently, if $v=|\Pcal|$, then
\begin{equation}\label{eq:retained-band}
       v=\left(b-a+o(1)\right)\frac{x}{\log x},
       \qquad
       \sum_{p\in\Pcal}\log p
          =\left(b-a+o(1)\right)x.
\end{equation}

If $L>a^2n$, then \eqref{eq:fixed-band-lower} is immediate.  We henceforth
assume $L\le a^2n$.  For every progression index $i$, define
\[
                         B_i:=\{p\in\Pcal:p\mid u+ic\}.
\]
We check the hypotheses of Lemma~\ref{lem:linear-cover}.

First, every pair of distinct primes $p,q\in\Pcal$ lies in some block:
indeed,
\[
                         pq\le b^2n\le n,
\]
so the $n$-divisor property supplies a progression term divisible by $pq$.

Second, two differently indexed blocks meet in at most one point.  If
distinct $p,q$ belonged to both $B_i$ and $B_j$, with $i\ne j$, then
\[
                         pq\mid(i-j)c.
\]
Since $p,q\nmid c$, this implies $pq\mid i-j$.  But
\[
       0<|i-j|<L\le a^2n\le pq,
\]
a contradiction.

Third, every block is proper.  If some $B_i$ were all of $\Pcal$, the
product of the distinct primes in $\Pcal$ would divide $u+ic$.  By
\eqref{eq:retained-band}, this would imply
\[
       \log(u+ic)
          \ge\sum_{p\in\Pcal}\log p
          =\left(b-a+o(1)\right)x,
\]
contrary to $\log(u+ic)\le\log H=o(x)$.

Finally, let $\Delta$ be the actual maximum point degree.  For
$p\in\Pcal$, the congruence
\[
                         u+ic\equiv0\pmod p
\]
selects one residue class of indices modulo $p$, because $p\nmid c$.
Therefore
\begin{equation}\label{eq:degree-bound}
       \Delta\le1+\frac{L}{a x}.
\end{equation}

Put $t=L/(a x)$.  Lemma~\ref{lem:linear-cover} and
\eqref{eq:degree-bound} give
\[
       v\le\Delta(\Delta-1)+1\le t^2+t+1.
\]
Since $v\to\infty$, this forces
\[
       t\ge\frac{\sqrt{4v-3}-1}{2}
        =(1-o(1))\sqrt v.
\]
Using \eqref{eq:retained-band} and $\log x=\tfrac12\log n$, we obtain
\begin{align*}
       L
       &\ge (a-o(1))x
             \sqrt{\frac{(b-a)x}{\log x}}\\
       &=\left(a\sqrt{2(b-a)}-o(1)\right)
             \frac{n^{3/4}}{\sqrt{\log n}},
\end{align*}
which proves \eqref{eq:fixed-band-lower}.
\end{proof}

\begin{proof}[Proof of Theorem~\ref{thm:quantitative}]
Take $b=1$ and $a=2/3$ in Lemma~\ref{lem:fixed-band}.  The coefficient in
\eqref{eq:fixed-band-lower} is
\[
       a\sqrt{2(b-a)}
          =\sqrt{\frac{8}{27}}.
\]
The use of the endpoint $b=1$ is legitimate: for distinct primes
$p,q\le\sqrt n$, one still has $pq\le n$, exactly the range covered by the
$n$-divisor property.  The lemma now gives \eqref{eq:main-lower-bound}
directly.
\end{proof}

\section{Consequences and scope}

\begin{proof}[Proof of Corollary~\ref{cor:exponent-exclusion}]
Suppose, to the contrary, that qualifying progressions exist along an
unbounded sequence of $n$.  Since $\alpha<1/2$,
\[
       \log H_n\le n^{\alpha+o(1)}=o(\sqrt n),
\]
so Theorem~\ref{thm:quantitative} applies.  On the other hand,
$\beta<3/8$ gives
\[
       n^{2\beta+o(1)}
          =o\!\left(\frac{n^{3/4}}{\sqrt{\log n}}\right),
\]
contradicting the theorem's lower bound for $L_n$.
\end{proof}

At the one-third point this gives the promised unconditional conclusion:
there is no unbounded sequence of positive $n$-divisor arithmetic
progressions with
\[
       L_n\le n^{2/3+o(1)},
       \qquad
       H_n\le\exp\!\bigl(n^{1/3+o(1)}\bigr).
\]
Natural floor or ceiling choices in the number of terms are absorbed by the
$o(1)$.  Indexing the progression from $0$ or from $1$ merely changes its
initial term.

\begin{remark}[Positivity and zero]
The $n$-divisor property in Definition~3.1 of Umans and Wang is explicitly
defined for sets of positive integers.  Thus zero is not an admissible
progression term or divisor witness here.  Allowing zero as a witness would
make the condition vacuous, since every positive integer divides zero.  A
constant progression reduces to the singleton case treated in the proof.
\end{remark}

\begin{remark}[What has and has not been disproved]
Proposition~3.4 of Umans and Wang \cite{UmansWang2025} shows that the
Arithmetic Progression Version implies the Strong
$(\alpha,\beta)$-Divisor Conjecture; no converse is asserted.
The result above consequently does \emph{not} disprove the full Strong
$(1/3,1/3)$-Divisor Conjecture.  A difference of two higher-rank generalized
arithmetic progressions need not be a one-dimensional progression.  The
decisive step above uses
\[
       (u+ic)-(u+jc)=(i-j)c.
\]
For a rank-two progression, the corresponding difference contains two
independent coefficients, and the at-most-one-intersection argument does
not follow.
\end{remark}

\begin{remark}[Endpoints]
Theorem~\ref{thm:quantitative} does not cover the case
$\log H=O(\sqrt n)$, and Corollary~\ref{cor:exponent-exclusion} makes no
claim when $\beta=3/8$.  Indeed, the lower bound
$L\gg n^{3/4}/\sqrt{\log n}$ does not exclude $L\le n^{3/4}$.  The constant
$\sqrt{8/27}$ is the outcome of the fixed-band argument and is not claimed
to be optimal.
\end{remark}

\end{document}